\documentclass{amsart}

\usepackage{amssymb}
\usepackage[all]{xy}
\usepackage{url}

\theoremstyle{plain}
\newtheorem{prop}{Proposition}
\newtheorem{thm}[prop]{Theorem}
\newtheorem*{thmb}{Theorem}

\newtheorem{fact}[prop]{Fact}

\theoremstyle{definition}

\theoremstyle{remark}
\newtheorem{rem}[prop]{Remark}

\numberwithin{prop}{section}
\numberwithin{example}{section} 
\numberwithin{claim}{prop}
\numberwithin{step}{prop}
\numberwithin{equation}{section}

\newcommand{\F}{\mathbb{F}}
\newcommand{\E}{\mathbb{E}}
\newcommand{\caO}{\mathcal{O}}

\newcommand{\rdn}{\mathrm{rad}}
\newcommand{\Lie}{\mathfrak{L}}
\newcommand{\euB}{\mathfrak{B}}
\newcommand{\euI}{\mathfrak{I}}
\newcommand{\euH}{\mathfrak{H}}

\newcommand{\euW}{\mathfrak{W}}
\newcommand{\eut}{\mathfrak{t}}
\newcommand{\eus}{\mathfrak{s}}
\newcommand{\env}{\mathrm{env}}

\newcommand{\Der}{\mathrm{Der}}
\newcommand{\ad}{\mathrm{ad}}
\newcommand{\der}{\partial}

\newcommand{\Irr}{\mathrm{Irr}}
\newcommand{\soc}{\mathrm{soc}}
\newcommand{\idn}{\mathrm{ind}}
\newcommand{\rst}{\mathrm{res}}
\newcommand{\uu}{\underline{u}}

\newcommand{\Hom}{\mathrm{Hom}}
\newcommand{\image}{\mathrm{im}}
\newcommand{\kernel}{\mathrm{ker}}
\newcommand{\eps}{\varepsilon}
\newcommand{\spn}{\mathrm{span}}

\newcommand{\argu}{\hbox to 7truept{\hrulefill}}

\begin{document}
\title[Restricted Zassenhaus algebras in characteristic $2$]{The projective indecomposable modules for the restricted Zassenhaus algebras in characteristic~$2$}

\author{B.~Lancellotti}
\address{Dipartimento di Matematica e Applicazioni\\
 Universit\`a degli Studi di Milano-Bicocca\\
  Via Roberto Cozzi 53, I-20125 Milano, Italy}
\email{b.lancellotti@campus.unimib.it}
\author{Th.~Weigel}
\address{Dipartimento di Matematica e Applicazioni\\
 Universit\`a degli Studi di Milano-Bicocca\\
  Via Roberto Cozzi 53, I-20125 Milano, Italy}
\email{thomas.weigel@unimib.it}



\date{February 11, 2015}

\begin{abstract}
It is shown that for the restricted Zassenhaus algebra
$\euW=\euW(1;n)$, $n>1$, defined over an algebraically closed field $\F$ of characteristic $2$ any projective indecomposable restricted $\euW$-module has maximal possible dimension $2^{2^n-1}$, and thus is isomorphic to some induced module $\idn^{\euW}_{\eut}(\F(\mu))$
for some torus of maximal dimension $\eut$. This phenomenon is in contrast to the
behavior of finite-dimensional non-solvable restricted Lie algebras in characteristic $p>3$
(cf. \cite[Theorem~6.3]{fsw:max}).
\end{abstract}


\maketitle


\section{Introduction}
\label{s:intro}
Let $\Lie$ be a finite-dimensional restricted Lie algebra defined over
an algebraically closed field $\F$ of characteristic $p>0$, and let
$\eut\subseteq\Lie$ be a torus of maximal dimension.
Then, as $\uu(\eut)$ - the {\it restricted universal enveloping algebra} of $\eut$ -
is a commutative and semi-simple associative $\F$-algebra, every irreducible restricted $\eut$-module is $1$-dimensional and also projective. Moreover, there exists a canonical one-to-one correspondence between the isomorphism classes of irreducible restricted $\eut$-modules and the set
\begin{equation}
\label{eq:dualt}
\eut^\circledast=\spn_{\F_p}\{\,t\in\eut\mid t^{[p]}=t\,\}^\ast,
\end{equation}
where $\F_p\subseteq\F$ denotes the prime field, and $\argu^\ast=\Hom_{\F_p}(\argu,\F_p)$. For $\mu\in\eut^\circledast$
let $\F(\mu)$ denote the corresponding irreducible restricted $\eut$-module. Then
\begin{equation}
\label{eq:proj}
P(\mu)=\idn_{\eut}^{\Lie}(\F(\mu))=\uu(\Lie)\otimes_{\uu(\eut)}\F(\mu)
\end{equation}
is a projective restricted $\Lie$-module (cf. \cite[Proposition~2.3.10]{weib:hom}), 
but in general it will be decomposable. Indeed, it has been shown in \cite[Theorem~6.3]{fsw:max} that for $p>3$ the finite-dimensional restricted Lie algebra $\Lie$ is solvable if, and only if, $P(0)$ is indecomposable. The authors conclude their paper with the remark, that for $p=2$ the ``if'' part of the assertion is false, e.g.,
for $p=2$ and $\Lie$ equal to the restricted Zassenhaus algebra $\euW(1;2)$ one has that $P(0)$ is indecomposable and coincides with the projective cover of the trivial $\euW(1;2)$-module. However, 
$\euW(1;2)$ is a non-solvable restricted Lie algebra. The main purpose of this note is to extend this result to all restricted Zassenhaus algebras $\euW(1;n)$, $n>1$, in characteristic $2$ and to all projective restricted modules $P(\mu)$, $\mu\in\eut^\circledast$.

\begin{thmb}
Let $n\geq 1,$ let $\euW(1;n)$  be a restricted Zassenhaus algebra
defined over an algebraically closed field of characteristic $2$, and let $\eut\subseteq\euW(1;n)$ be a torus of maximal dimension. Then $P(\mu)$ is indecomposable for all $\mu\in\eut^\circledast$. In particular, all projective indecomposable restricted
$\euW(1;n)$-modules have maximal possible dimension $2^{2^n-1}$, and $\euW(1;n)$ has maximal $0$-p.i.m.
\end{thmb}

The proof of the theorem will be arranged in two major steps. First, it will be shown that every restricted Zassenhaus algebra $\euW(1;n)$ has a direct sum decomposition
\begin{equation}
\label{eq:deco}
\euW(1;n)=\eut_0\oplus\euB,
\end{equation}
where $\eut_0$ is a torus of maximal dimension, and $\euB$
is a certain restricted Lie subalgebra of $\euW(1;n)$.
In the second step it will be shown that the trivial $\euW(1;n)$-module $\F$ and $L= W(1;n)^{(1)}$
are the only isomorphism types of irreducible restricted $\euW(1;n)$-modules. From this fact and \cite[Lemma~6.1]{fsw:max} one concludes that $\euW(1;n)$ has {\it maximal $0$-p.i.m.}, i.e., the projective cover $P_\F$ of the trivial $\euW(1;n)$-module has maximal possible dimension
$2^{2^n-1}$, and thus must be isomorphic to $P(0)$ (cf. \eqref{eq:proj}).
An elementary calculation then shows that the projective cover $P_L$ of the restricted $\euW(1;n)$-module $L$ must be isomorphic to
$P(\mu)$ for $\mu\in\eut^\circledast\setminus\{0\}$, i.e., 
any projective indecomposable restricted $\euW(1;n)$-module has
the maximal possible dimension $2^{2^n-1}$ (cf. Remark~\ref{rem:final}). As a by-product we may also conclude that
the restricted Lie subalgebra $\euB$ has properties similar to the Borel subalgebra in a classical restricted Lie algebra (cf. Remark~\ref{rem:borel} and Theorem~\ref{thm:witt2}(b)).


\section{Preliminaries}
\label{s:llie}
From now on we will assume that $\F$ is a field of characteristic $p>0$.

\subsection{Simple algebras and simple restricted Lie algebras}
\label{ss:simp}
A (restricted) $\F$-Lie algebra $\Lie$ is said to be {\it simple}
if any (restricted) Lie ideal $\euI$ of $\Lie$ coincides
either with $\Lie$ or with $0$. The following fact is straightforward and its easy proof is left to the reader (cf. \cite[\S 6]{lanc:tesi}). 

\begin{fact}
\label{fact:simp} 
Let $\F$ be a field of characteristic $p>0$.

\noindent
{\rm (a)} Let $\Lie$ be a non-abelian finite-dimensional simple restricted $\F$-Lie algebra, and let
$L=\soc_{\Lie}(\Lie)$ denote the socle of the adjoint $\Lie$-module. Then $L$ is a minimal Lie ideal of $\Lie$ and a non-abelian simple $\F$-Lie algebra.
Moreover, $\Lie$ coincides with the minimal $p$-envelope 
$\env_p(L)$ of $L$.

\noindent
{\rm (b)} Let $L$ be a non-abelian finite-dimensional simple $\F$-Lie algebra. Then the minimal $p$-envelope $\env_p(L)$ is a finite-dimensional
simple restricted $\F$-Lie algebra, and one has $L=\soc_{\env_p(L)}(\env_p(L))$.
\end{fact}


\subsection{Generalized Borel subalgebras}
\label{ss:genbor}
Let $\Lie$ be a finite-dimensional restricted Lie algebra over an algebraically closed
field $\F$ of characteristic $p>0$. By $\Irr_p(\Lie)$ we will denote the set of isomorphism classes of restricted irreducible $\Lie$-modules. A proper restricted Lie subalgebra 
$\euB$ of $\Lie$ will be said to be a {\it generalized Borel subalgebra}, if
\begin{itemize}
\item[(i)] the unipotent radical $\rdn_{u}(\euB)$ of $\euB$ is non-trivial, i.e.,
$\rdn_{u}(\euB)\not=0$;
\item[(ii)] for any irreducible restricted $\Lie$-module $S$, $[S]\in\Irr_p(\Lie)$, the restricted $\euB/\rdn_{u}(\euB)$-module $S^{\rdn_{u}(\euB)}$ is irreducible;
\item[(iii)] the map $\chi_{\euB}\colon\Irr_p(\Lie)\to\Irr_p(\euB/\rdn_{u}(\euB))$ given by $\chi_{\euB}([S])=[S^{\rdn_{u}(\euB)}]$ is injective.
\end{itemize}

\begin{rem}
\label{rem:borel}
Let $\Lie$ be a classical finite-dimensional simple restricted Lie algebra, and let $\euB$ be a Borel subalgebra of $\Lie$. It is well known that for any irreducible restricted $\Lie$-module $S$, $S^{\rdn_u(\euB)}$ is $1$-dimensional. Moreover the mapping
$\chi_{\euB}\colon\Irr_p(\Lie)\to\Irr_p(\euB/\rdn_{u}(\euB))$ is injective. Hence
$\euB$ is a generalized Borel subalgebra of $\Lie$.
\end{rem}

\begin{rem}
\label{rem:veis}
A similar generalization of the concept of Borel subalgebra to $\mathbb{Z}$-graded restricted Lie algebras, and especially to restricted Lie algebra of Cartan type is given in \cite[Theorem~4(a) and (c)]{veis}.
\end{rem}

\section{The restricted Zassenhaus algebra $\euW(1;n)$}
\label{ss:witt}
Let $\F$ be a field of characteristic $p>0$,
and let $\E$ be a finite subfield of $\F$. The Lie algebra
\begin{equation}
\label{eq:zas1}
L(\E)=\spn_{\F}\{\,u_a\mid a\in \E\,\}
\end{equation}
with bracket given by $[u_a,u_b]=(b-a)\,u_{a+b}$, $a,b\in\E$,
is called the {\it Zassenhaus algebra} with respect to $\E$ (cf. \cite[p.~47ff.]{zass:lie}).

Let $\caO(1;n)$ denote the truncated divided power $\F$-algebra
of dimension $p^{n}$ with basis $\{\,x^{(k)}\mid 0\leq k\leq p^{n-1}\}$.
The {\it Witt algebra} $W(1;n)$ is the Lie subalgebra of $\Der(\caO(1;n))$
of {\it special derivations}, i.e.,
\begin{equation}
\label{eq:zas2}
W(1;n)=\spn_{\F}\{\,e_j=x^{(j+1)}\cdot\der\mid -1\leq j\leq p^n-2\,\},
\end{equation}
where $\partial$ denotes the derivation defined by
\begin{equation}\label{part}
\partial (x^{(a)})=
\begin{cases} 0, & \mbox{if }a=0,\\
x^{(a-1)}, & \mbox{if }1\leq a\leq p^{n}-1.
\end{cases} 
\end{equation}
In particular, one has
\begin{equation}
\label{eq:stcon}
[e_{i},e_{j}]=
\begin{cases}
c_{ij}\cdot e_{i+j} & \text{if $-1\leq i\leq p^{n}-2$,}\\
\hfil 0\hfil&\text{otherwise,}
\end{cases}
\end{equation}
where $c_{ij}=\binom{i+j+1}{j}-\binom{i+j+1}{i}$.
It is well known that if $\E\subseteq\F$ is the finite field with $p^n$ elements and if $\F$ is perfect, then one has
a (non-canonical) isomorphism
\begin{equation}
\label{eq:zas3}
L(\E)\simeq W(1;n)
\end{equation}
(cf. \cite{bemis:alza}, \cite{lanc:tesi}, or also \cite[Theorem~7.6.3(1)]{strade:book1} for $p$ odd). If $p>2$, then $W(1;n)$ is a simple Lie algebra (cf. \cite[Theorem~4.2.4]{strafar:book}); while if $p=2$ and $n>1$ then
\begin{equation}
\label{eq:zas4}
W(1;n)^{(1)}=[W(1;n),W(1;n)]=\spn_{\F}\{\,e_j\mid -1\leq j\leq 2^n-3\,\}
\end{equation}
is simple (cf. \cite{dz}).
The minimal $p$-envelope $\euW(1;n)=\env_{p}(W(1;n))$ of the Zassenhaus algebra $W(1;n)$ coincides with the deriviation algebra of $W(1;n)$
(cf. \cite[Theorem~7.1.2(2) and Theorem~7.2.2(1)]{strade:book1}) and is also called the {\it restricted Zassenhaus algebra}, i.e.,
identifying $W(1;n)$ with its image in the Lie algebra $\Der(W(1;n))$ one has
\begin{equation}
\label{eq:derwitt}
\euW(1;n)=\bigoplus_{1\leq k\leq n-1} \F\cdot \der^{[p]^k}\oplus W(1;n).
\end{equation}
The following result holds.

\begin{prop} 
\label{prop:env}
If $p=2$ and $n>1$, then $\euW(1;n)$ is isomorphic to the minimal $2$-envelope of the simple Lie algebra $W(1;n)^{(1)}$.
\end{prop}

\begin{proof}
Put $\euW=\euW(1;n)=\Der(W(1;n))$ and consider the canonical injective map
\begin{equation}
\label{eq:canma}
\alpha\colon 
\xymatrix{
W(1;n)^{(1)}\ar[r]^-j& W(1;n)\ar[r]^-{\ad_W}&\euW}.
\end{equation} 
As
$C_{\euW}(\der)=\spn_{\F}\{\,\der^{[2]^j}\mid 0\leq j\leq n-1\,\}$, one has
that $C_\euW(A)=0$ for $A=\image(\alpha)$.
Hence, by Fact~\ref{fact:simp}, $\euH=\langle A\rangle_p$ - the restricted Lie subalgebra of $\euW$
generated by $A$ - is a minimal $p$-envelope of $W(1;n)^{(1)}$
(cf. \cite[Theorem~1.1.7]{strade:book1}). For simplicity we assume that $\alpha$ is given by inclusion. By construction, $\bigoplus_{1\leq k\leq n-1} \F\cdot \der^{[p]^k}\oplus W(1;n)^{(1)}$
is an $\F$-subspace of $\euH$.
An elementary calculation shows that $e_{2^{n-1}-1}^{[2]}=e_{2^n-2}$
(cf. \cite[\S 1]{joe:coh}). Thus $\euH=\euW$, and this yields the claim.
\end{proof}

\subsection{Toral complements}
\label{ss:torcom} 
We define the restricted Lie subalgebra $\euB$ of $\euW(1;n)$ by
\begin{equation}
\label{eq:zas5}
\euB=\spn_{\F}\{\,e_j\mid 0\leq j\leq p^n-2\,\},
\end{equation}
i.e., $\euB\subseteq W(1;n)$.

\begin{prop}
\label{prop:witt}
Let $\F$ be a perfect field of characteristic $p$. 
\begin{itemize}
\item[(a)] Let
$\eut\subseteq\euW(1;n)$ be any torus of maximal dimension. Then $\eut$ has dimension $n$.
\item[(b)] There exists a torus $\eut_0\subset \euW(1;n)$ of maximal dimension such that 
\begin{equation}
\label{eq:max1}
\euW(1;n)=\eut_0\oplus\euB.
\end{equation}
\end{itemize}
\end{prop}

\begin{proof}
(a) For $p>2$ the assertion has been shown in \cite[Theorem~7.6.3]{strade:book1}. 
Hence we may assume that $\F$ is a perfect field of characteristic $2$. 
Define $s=e_{-1}+e_{2^{n}-2}\in W(1;n)$. By induction, one concludes that
\begin{align}
s^{2^{k}}&=\partial^{2^{k}}+e_{2^{n}-2^{k}-1}
\qquad\text{for $k\in\{1,\ldots,n-1\}$, and}\label{eq:spow1}\\
s^{2^{n}}&=\partial^{2^{n}}+e_{-1}+e_{2^{n}-2}=s.\label{eq:spow2}
\end{align}
In particular, the minimal polynomial of $\ad_W(s)$ divides the separable polynomial $T^{2^{n}}-T\in\F[T]$. Thus $\ad_W(s)\in\euW(1;n)$ is semi-simple and the elements 
$\ad_W(s)^{2^{i}}\in\Der(W(1;n))=\euW(1;n)$, $0\leq i\leq n-1$, are linearly independent.
Hence
$\eut_0=\spn_{\mathbb{F}}\{\,\ad(s)^{2^{i}}\mid 0\leq i\leq n-1\,\}$ is a torus of 
$\euW(1;n)$ of dimension $n$, i.e., the maximal dimension  $\mathrm{MT}(\euW(1;n))$ of a torus
of $\euW(1;n)$ must be greater or equal to $n$. 
Moreover the $\F$-algebras $\caO(1;n)$ and $\caO(n;\underline{1})$ are isomorphic, so there exists an injective homomorphism of restricted Lie algebras $i\colon \euW(1;n)\hookrightarrow \euW(n;\underline{1})$. Since $C(\euW(1;n))=0$  (cf. Proposition~\ref{prop:env} and \cite[Proposition~1.4]{joe:coh}), it follows from \cite[Lemma~ 1.2.6(1)]{strade:book1} that $\mathrm{MT}(\euW(1;n))=\mathrm{TR}(W(1;n))$, and thanks to \cite[Theorem~1.2.7(1)]{strade:book1} this yields 
$\mathrm{MT}(\euW(1;n))\leq \mathrm{MT}(\euW(n;\underline{1}))=n$ (cf. \cite[Corollary~7.5.2]{strade:book1}). So we may conclude that the maximal dimension $\mathrm{MT}(\euW(1;n))$ of a torus of $\euW(1;n)$ is equal to $n.$ 

\noindent
(b) For $p=2$ the just mentioned argument shows that $\eut_0\oplus\euB=\euW(1;n)$.
For $p>2$ one may identify $W(1;n)$ with $L(\E)$, where $\E\subseteq\F$, $|\E|=p^n$, and define 
$\eut_0=\spn_{\F}\{\, u_0^{[p]^j}\mid 0\leq j\leq n-1\,\}$ 
(cf. the proof of \cite[Theorem~7.6.3(2)]{strade:book1}).
\end{proof}


\subsection{The restricted subalgebra $\euB$}
\label{ss:subB}
Note that $e_0\in\euB$ is a toral element. Moreover, as
$\rdn_{u}(\euB)=\spn_{\F}\{\,e_j\mid 1\leq j\leq p^n-2\,\}$, one has
\begin{equation}
\label{eq:max2}
\euB=\eus\oplus \rdn_{u}(\euB),
\end{equation}
where $\eus=\F\cdot e_0$. For $\lambda\in\{0,\ldots,p-1\}$ we denote by
$\F[\lambda]$ the $1$-dimensional restricted $\euB$-module satisfying
\begin{equation}
\label{eq:max3}
e_0\cdot z=\lambda\cdot z \qquad\text{for $z\in\F[\lambda]$.}
\end{equation}
In particular, $x\cdot z=0$ for all $x\in\rdn_u(\euB)$ and $z\in\F[\lambda]$.
For our purpose the following property will turn out to be useful (cf. \cite[Proposition~4.7(1) and (2)]{joe:coh}).

\begin{prop}
\label{prop:repwitt}
Let $\F$ be a field of characteristic $p>0$ and let $\euB$ be as described above.
Then one has isomorphisms of restricted $\euW(1;n)$-modules
\begin{align}
\idn_{\euB}^{\euW(1;n)}(\F[-2])&\simeq W(1;n),\label{eq:rep1}\\
\idn_{\euB}^{\euW(1;n)}(\F[1])&\simeq W(1;n)^\ast=\Hom_{\F}(W(1;n),\F).\label{eq:rep2}
\end{align}
\end{prop} 

\begin{rem}
\label{rem:L}
By Fact~\ref{fact:simp}(a), $W(1;n)$ is a $\euW(1;n)$-submodule of the adjoint 
$\euW(1;n)$-module containing $L=\mathrm{soc}_{\euW(1;n)}(\euW(1;n))$. Hence $L$ is equal to $\mathrm{soc}_{\euW(1;n)}(W(1;n))$. As $\dim(W(1;n)/L)$  is equal to $1$ for $p=2$
and $0$ otherwise, this shows that $L$ is the unique maximal $\euW(1;n)$-submodule of $W(1;n)$.
\end{rem}


\subsection{Projective indecomposable modules}
\label{ss:pims}
Proposition~\ref{prop:witt} has the following consequence for $p=2$ .

\begin{thm}
\label{thm:witt2}
Let $\F$ be an algebraically closed field of characteristic $2$, let $n\geq 1$.
\begin{itemize}
\item[(a)] $\Irr_p(\euW(1;n))=\{\,[\F],[L]\,\}$, where $\F$ denotes the $1$-dimensional trivial 
$\euW(1;n)$-module, and $L=\soc_{\euW(1;n)}(\euW(1;n))=W(1;n)^{(1)}$ 
(cf. Remark~\ref{rem:L}).
\item[(b)] If $n\geq 2$, $\euB$ is a generalized Borel subalgebra of $\euW(1;n)$.
\item[(c)] Let $\eut_0$ be a torus of maximal dimension of $\euW(1;n)$
satisfying $\euW(1;n)=\eut_0\oplus\euB$ (cf. \eqref{eq:max1}). Then
$L^{\eut_0}=0$. In particular, $\euW(1;n)$ is a restricted Lie algebra with maximal $0$-p.i.m., i.e., if $P_{\F}$ is the projective cover of the
$1$-dimensional trivial $\euW(1;n)$-module, then one has
$P_{\F}\simeq\idn_{\eut_0}^{\euW(1;n)}(\F(0))$. Hence
\begin{equation}
\label{eq:maxpim}
\dim(P_{\F})=2^{\dim(\euW(1;n))-\dim(\eut_0)}=2^{2^n-1}.
\end{equation}
\item[(d)] Let $P_L$ denote the projective cover of the irreducible $\euW(1;n)$-module $L$. Then $\dim(P_L)=2^{2^n-1}$. In particular,
$P_L\simeq\idn_{\eut_0}^{\euW(1;n)}(\F(\mu))$ for any non-trivial irreducible
$\eut_0$-module $\F(\mu)$, $\mu\in\eut_0^\circledast \setminus\{0\}$.
\end{itemize}
\end{thm}

\begin{proof} 
(a) Let $\euW=\euW(1;n)$.  Since $\euB/\rdn_u(\euB)\simeq\eus$, one has
\begin{equation}
\Irr_p(\euB)=\{[\F[0]],[\F[1]]\}.
\end{equation}
Let $S$ be an irreducible restricted $\euW$-module, and let
$\Sigma=\soc_{\euB}(\rst^{\euW}_{\euB}(S))$.
Then $\Sigma$ contains either a restricted $\euB$-submodule isomorphic
to $\F[0]$ or $\F[1]$ or both. As
\begin{equation}
\label{eq:maxpim2}
\Hom_{\euB}(\F[\lambda],\rst^{\euW}_{\euB}(S))\simeq
\Hom_{\euW}(\idn_{\euB}^{\euW}(\F[\lambda]),S),\qquad \lambda\in\{0,1\},
\end{equation}
Proposition~\ref{prop:repwitt} implies that 
$S$ is either a homomorphic image of 
$\idn_{\euB}^{\euW}(\F[0])$, in which case $S\simeq W/L\simeq \F$ (cf.~Remark~\ref{rem:L}), or
$S$ is a homomorphic image of 
$\idn_{\euB}^{\euW}(\F[1])$, in which case $S\simeq L^\ast$.
Hence $\Irr_p(\euW)=\{\,[\F], [L^\ast]\,\}$. 
As $L$ is an irreducible restricted $\euW$-module, comparing dimensions yields that $L\simeq
L^\ast$ as restricted $\euW$-modules.

\noindent
(b) As $\euB/\rdn_u(\euB)\simeq\eus$ is a torus, $M^{\rdn_u(\euB)}$ is a semi-simple restricted $\euB/\rdn_u(\euB)$-module for any finite-dimensional $\euW$-module $M$. Hence
\begin{equation}
\label{eq:maxpim3}
M^{\rdn_u(\euB)}=\soc_{\euB}(\rst^{\euW}_{\euB}(M)).
\end{equation}
Obviously,
$\F^{\rdn_u(\euB)}=\F[0]$ and $L^{\rdn_u(\euB)}=\F\cdot e_{2^{n}-3},$ i.e., $L^{\rdn_u(\euB)}\simeq\F[1]$ as $\euB$-module.
So the mapping $\chi_\euB=[\argu^{\rdn_u(\euB)}]\colon
\Irr_p(\euW)\to\Irr_p(\euB/\rdn_u(\euB))$ is injective 
showing that $\euB$ is a generalized Borel subalgebra.

\noindent
(c) As $\euW=\eut_0\oplus\euB$, one has
\begin{equation}
\label{eq:maxpim4}
U=\rst^{\euW}_{\eut_0}(\idn_{\euB}^{\euW}(\F[0]))\simeq
\uu(\eut_0)\otimes_{\F}\F[0]\simeq\uu(\eut_0),
\end{equation}
i.e., $U$ is a free $\uu(\eut_0)$-module of rank $1$.
As $\idn_{\euB}^{\euW}(\F[0])$ is isomorphic to the restricted
$\euW$-module $W$ (cf. \eqref{eq:rep1}), one has a short exact sequence of restricted 
$\eut_0$-modules
\begin{equation}
\label{eq:maxpim5}
\xymatrix{
0\ar[r]& \rst^{\euW}_{\eut_0}(L)\ar[r]&
U\ar[r]&\F(0)\ar[r]&0,
}
\end{equation}
where $\F(0)$ denotes the $1$-dimensional trivial $\eut_0$-module, i.e.,
$\rst^{\euW}_{\eut_0}\!(\!L\!)$ is isomorphic to the
augmentation ideal $\kernel(\eps\colon\uu(\eut_0)\to\F)$ of
$\uu(\eut_0)$. Hence $L^{\eut_0}=0$.
Since any non-trivial, irreducible restricted $\euW$-module must be isomorphic to $L$, one concludes from \cite[Lemma~6.1]{fsw:max} that the projective cover $P_\F$ of the $1$-dimensional irreducible $\euW$-module $\F$ must be isomorphic to $\idn_{\eut_0}^{\euW}(\F(0))$.

\noindent
(d) As one has an isomorphism  $\uu(\euW)\simeq P_\F\oplus \dim(L)\cdot P_L$ of $\euW$-modules, one concludes
that
\begin{equation}
\label{eq:maxpim6}
\dim(P_L)=\frac{2^{2^n+n-1}-2^{2^n-1}}{2^n-1}=
\frac{2^{2^n-1}(2^n-1)}{2^n-1}=2^{2^n-1}.
\end{equation}
If $\F(\mu)$ is a non-trivial irreducible $\eut_0$-module,
then $\idn_{\eut_0}^{\euW}(\F(\mu))$ is projective, and has
dimension equal  to $2^{2^n-1}$. As $\F(\mu)$ is isomorphic to
a direct summand of $L$ (cf. \eqref{eq:maxpim2}),
$P_L$ is a homomorphic image of $\idn_{\eut_0}^{\euW}(\F(\mu))$.
Since both restricted $\euW$-modules have the same dimension, they must be isomorphic. This completes the proof.
\end{proof}

\begin{rem}
\label{rem:n}
In particular, Theorem \ref{thm:witt2}(b) is also true for $n=1$ if in the definition of a generalized Borel subalgebra condition (i) is omitted.
\end{rem}

\begin{rem}
\label{rem:final}
Let $\F$ be an algebraically closed field of characteristic $p$, let $\Lie$ be a finite-dimensional restricted Lie algebra, and let $\eut\subseteq\Lie$ be a torus of maximal dimension. 
If $P$ is a projective indecomposable restricted $\Lie$-module, then $P$ is isomorphic to the projective cover $P_S$ for some irreducible restricted $\Lie$-module $S$.
Let $\F(\mu)$ be an irreducible restricted $\eut$-submodule of $\rst^{\Lie}_{\eut}(S)$.
Then - as $\Hom_{\Lie}(\idn_{\eut}^{\Lie}(\F(\mu)),S)\not=0$ - $P$ is isomorphic to a direct summand of
$P(\mu)=\idn_{\eut}^{\Lie}(\F(\mu))$ which is a projective restricted $\Lie$-module for any $\mu\in\eut^\circledast$. Hence 
\begin{equation}
\label{eq:maxid}
\dim(P)\leq p^{\dim(\Lie)-\mathrm{MT}(\Lie)}.
\end{equation}
This shows that for $\Lie=\euW(1;n)$, $n>1$, and $p=2$, equality holds in
\eqref{eq:maxid} for any projective indecomposable restricted $\Lie$-module $P$.
Therefore, for the restricted Lie algebra $\euW$ all p.i.m.s have maximal possible dimension. 
\end{rem}


\providecommand{\bysame}{\leavevmode\hbox to3em{\hrulefill}\thinspace}
\providecommand{\MR}{\relax\ifhmode\unskip\space\fi MR }
\providecommand{\MRhref}[2]{%
  \href{http://www.ams.org/mathscinet-getitem?mr=#1}{#2}
}
\providecommand{\href}[2]{#2}

\end{document}